\documentclass[12pt]{article}
\usepackage{amsmath}
\usepackage{amssymb,latexsym}
\usepackage{color}
\usepackage{amsfonts}

\newcommand{\R}{\mathbb{R}}

\newenvironment{proof}{\noindent{\it Proof}\rm.}{\hfill $\Box$}
\newenvironment{proofof}[1]{\bigskip\noindent{\it Proof of~#1.}\rm}{\hfill $\Box$}
\newtheorem{theorem}{Theorem}[section]

\newtheorem{corollary}[theorem]{Corollary}
\newtheorem{proposition}[theorem]{Proposition}
\newtheorem{remarks}[theorem]{Remarks}
\newtheorem{remark}[theorem]{Remark}
\numberwithin{equation}{section}

\begin{document}

\title{Trace identities for commutators, with applications to the
distribution of eigenvalues}
\author{Evans M. Harrell II$^1$ \and Joachim Stubbe$^2$}
\date{\small
$^1$School of Mathematics,
Georgia Institute of Technology,\\
Atlanta GA 30332-0160, USA\\
{\tt harrell@math.gatech.edu}\\
$^2$EPFL, IMB-FSB\\
Station 8,\\
CH-1015 Lausanne, Switzerland\\
{\tt Joachim.Stubbe@epfl.ch}\\}

\maketitle

\begin{abstract}

We prove trace identities for commutators of operators, which are
used to derive sum rules and sharp universal bounds for the
eigenvalues of periodic Schr\"odinger operators and Schr\"odinger
operators on immersed manifolds.   In particular, we prove bounds
on the eigenvalue $\lambda_{N+1}$ in terms of the lower spectrum,
bounds on ratios of means of eigenvalues, and universal
monotonicity properties of eigenvalue moments, which imply sharp
versions of Lieb-Thirring inequalities. In the geometric context
we derive a version of Reilly's inequality bounding the eigenvalue
$\lambda_{N+1}$ of the Laplace-Beltrami operator 
on an immersed manifold of dimension $d$
by a universal
constant times $\|h\|_{\infty}^2 N^{2/d}$.


\end{abstract}

\eject
\section{Introduction}\label{intro}

In \cite{HaSt1} the authors derived ``sum rule'' identities and
sharp universal inequalities for certain self-adjoint operators
$H$, including the Dirichlet Laplacian on bounded Euclidean
domains and Schr\"odinger operators with discrete spectra.  The
essential idea was to exploit algebraic relations among the first
and second commutators of $H$ with an auxiliary self-adjoint
operator $G$. (Notationally, the commutator is given by $[H, G] :=
HG - GH$, and by the first and second commutators we refer to
$[H,G]$ and to $\left[G, \left[H, G\right]\right]$.)  In the
canonical case $H$ was of the form $- \Delta + V({\bf x})$, where
$- \Delta$ designates the Dirichlet Laplacian on a bounded domain
$\Omega$, and $G$ was chosen as a Euclidean coordinate function.
Due to the Dirichlet condition, the product of $G$ with the
eigenfunctions of $H$ remained in the domain of definition of $H$,
allowing manipulation of products of operators without much
needing to address technical questions of the domains of
definition of partial differential operators.  The trace
identities of \cite{HaSt1} were found to imply and unify several
known inequalities for the spectra of Laplacians as well as new
inequalities, which in many cases were shown to be optimal.  Among
the many more recent related articles are some establishing
connections with ``semiclassical'' spectral analysis, for example
by making connections with Lieb-Thirring inequalities and
inequalities that are sharp in the Weyl limit
\cite{HaHe1,HaHe2,Stu}.

The original intent of the present article
was to use similar methods to study the spectrum
of periodic Schr\"odinger operators, a case that has been much less considered
from the point of view of universal spectral bounds.
One reason for the lack of attention to periodic problems is that
multiplication by a coordinate function does not preserve a core
of the domain of self-adjointness of $H$, so the simple
algebraic relations in the canonical case for $[H,G]$ and
$\left[G, \left[H, G\right]\right]$ are not valid. We have
therefore sought an alternative whereby $H$ is commuted with a
family of auxiliary operators $G$ not assumed to be self-adjoint.
Indeed the case of greatest interest will be when $G$ is the
unitary operator of multiplication by $\exp(- i {\bf q} \cdot {\bf
x})$.  The second section of this article contains an abstract
trace formula of this sort (see Theorem
\ref{abstract-trace-general-theorem}), which exhibits useful
simplifications when $G$ is unitary.

The identity that forms the point of departure for the later parts
of the article 
turned out to have
the same algebraic form as one that applies to
Schr\"odinger operators on embedded manifolds with bounded mean
curvature \cite{Har}, and it was consequently possible to derive
spectral bounds in the two situations simultaneously.  This is
done in Sections 3 and \ref{UniversalRiesz} for the unperturbed
Laplacian and for Schr\"{o}dinger operators with bounded
potentials. In particular we show that Riesz means still possess a
monotonicity property similar to the one first discovered in
\cite{HaHe1} for Dirichlet Laplacians on bounded domains in
$\mathbb{R}^d$ with, however, a constant shift in its argument
depending on the geometry (see Theorem
\ref{monotone1}). In Section \ref{LTSection} we adapt the method
of \cite{Stu} to prove a universal monotonicity property of
Riesz means for periodic Schr\"{o}dinger operators and
Schr\"{o}dinger operators on manifolds of bounded mean curvature,
which implies sharp Lieb-Thirring inequalities (see Theorem
\ref{LT}). Special cases include sharp Lieb-Thirring
inequalities for Schr\"{o}dinger operators on spheres
$\mathbb{S}^d\subset \mathbb{R}^{d+1}$, which have been studied
previously with applications to Navier-Stokes equations
\cite{Il1},\cite{Il3}. In Section 6 we discuss applications of our
abstract trace identity of Theorem
\ref{abstract-trace-general-theorem} when commuting with unitary
operators which provides a new proof of the trace inequality of
\cite{Har} on manifolds of bounded mean curvature. We then prove a
new Reilly-type bound on eigenvalues optimal in the asymptotic
behavior. Finally, in Section 7 we provide some simple explicit
examples illustrating the optimality of our results, 
including some results on the distribution of
lattice points.

\section{Some trace identities and their consequences}\label{IDSection}

On a Hilbert space $\mathcal{H}$ we consider a self-adjoint
operator $H$ with domain of definition $\mathcal{D}_H$, along with
a second linear operator $G$ subject to some conditions relating
to $\mathcal{D}_H$. In many of the examples to be discussed in
this article the spectrum of $H$ consists entirely of eigenvalues
$\lambda_j$, and the corresponding eigenfunctions $\phi_j$ are
chosen to form an orthonormal basis of the underlying Hilbert
space $\mathcal{H}$. (The extension needed if $H$ has continuous
spectrum is not difficult, and has been explicitly presented in
the case where $G$ is self-adjoint in \cite{HaSt2}.) Although this
result is a special case of a more general trace identity
based only on algebraic properties of operators, we
present first of all a version assuming that $H$ has purely
discrete spectrum:

\begin{theorem}\label{AbstrTrace}
Let $H$ be a self-adjoint operator
on $\mathcal{H}$, with purely discrete spectrum.
Let $G$ be a linear operator with domain
$\mathcal{D}_G$ and adjoint $G^*$ defined on $\mathcal{D}_{G^*}$
such that $G(\mathcal{D}_H)\subseteq \mathcal{D}_H\subseteq
\mathcal{D}_G$ and $G^*(\mathcal{D}_H)\subseteq
\mathcal{D}_H\subseteq \mathcal{D}_{G^*}$, respectively.  Fix
a subset $J$ of the spectrum of $H$.  Then

\begin{equation}\label{tf2discrete}
\begin{split}
    &\;\frac1{2}\sum_{\lambda_j\in J}  (z-\lambda_j)^2\,\big(\langle[G^*,[H,G]]\phi_j,\phi_j\rangle+\langle[G,[H,G^*]]\phi_j,\phi_j\rangle\big)\\
    &-\sum_{\lambda_j\in
    J}(z-\lambda_j)\,\left(\|[H,G]\phi_j\|^2+\|[H,G^*]\phi_j\|^2\right)\\
    &=\\
    &\sum_{\lambda_j\in J}\sum_{\lambda_k\notin J}
    (z-\lambda_j)(z-\lambda_k)(\lambda_k-\lambda_j)\big(|\langle G\phi_j,\phi_k\rangle|^2+|\langle G^*\phi_j,\phi_k\rangle|^2\big).\\
    \end{split}
\end{equation}
\end{theorem}

\noindent
If $G=G^*$, then \eqref{tf2discrete} reduces to a key identity
used in \cite{HaSt1,HaSt2} as a step towards the trace
identities introduced there.

In preparation for the proof of Theorem \ref{AbstrTrace} and the
presentation of the general trace identity, we collect some
straightforward algebraic identities. Let $P$ be a spectral
projector of the self-adjoint operator $H$, and define the pair of
operators $A$ and $B$ by
\begin{equation}\label{op1}
    A=(1-P)GP,\quad B=PG(1-P).
\end{equation}
We recall the standard inclusions
$\text{Ran}\,(A^*A)\subseteq \text{Ran}\,(P)$,
$\text{Ran}\,(AA^*)\subseteq \text{Ran}\,(1-P)$,
$\text{Ran}\,(B^*B)\subseteq \text{Ran}\,(1-P)$ and
 $\text{Ran}\,(BB^*)\subseteq \text{Ran}\,(P)$.

\begin{theorem} \label{abstract-trace-general-theorem} Let $H$ be a self-adjoint operator
on $\mathcal{H}$ and $P$ be a spectral projector of $H$. Let $G$
be a linear operator with domain $\mathcal{D}_G$ and adjoint $G^*$
defined on $\mathcal{D}_{G^*}$ such that
$G(\mathcal{D}_H)\subseteq \mathcal{D}_H\subseteq \mathcal{D}_G$
and $G^*(\mathcal{D}_H)\subseteq \mathcal{D}_H\subseteq
\mathcal{D}_{G^*}$, respectively. Then
\begin{equation}\label{tf1}
\begin{split}
   &
tr\,\big(H^2(G^*[H,G]+G[H,G^*])P\big)-tr\,\big(H([H,G^*][H,G]+[H,G][H,G^*])P\big)\\
   &=\\
 & tr\,\big(HA^*H^2A-HAH^2A^*\big)+tr\,\big(HBH^2B^*-HB^*H^2B\big).\\
   \end{split}
\end{equation}
\end{theorem}

\begin{proof}
Using the projectors $P$ and $1-P$ we write
$$tr\,\big(H^2(G^*[H,G]+G[H,G^*])\big)$$
in the form
$$
tr\,\big(H^2(G^*P[H,G]+GP[H,G^*])P\big)+
     tr\,\big(H^2(G^*(1-P)[H,G]+G(1-P)[H,G^*])P\big),
$$
the first term of which can be computed as
\begin{equation*}
\begin{split}
     & tr\,\big(H^2(G^*P[H,G]+GP[H,G^*])P\big) \\
     &=tr\,\big(H^2G^*P(HG-GH)P+H^2GP(HG^*-G^*H)P\big)\\
&=tr\,\big(H(HG^*-G^*H)P(HG-GH)P+H(HG-GH)P(HG^*-G^*H)P\big)\\&\quad +tr\,\big(HG^*HP(HG-GH)P+HGHP(HG^*-G^*H)P\big).\\
\end{split}
\end{equation*}
The final term
in this expression
vanishes thanks to the cyclic property of the trace
({\it viz.}, $tr(AB) = tr(BA)$),
and therefore
\begin{equation*}
tr \, \big(H^2(G^*P[H,G]+GP[H,G^*])P\big)=tr\,\big(H([H,G^*]P[H,G]+[H,G]P[H,G^*])P\big).
\end{equation*}
Adding and subtracting the expression
$tr\,\big(H([H,G^*](1-P)[H,G]+[H,G](1-P)[H,G^*])P\big)$, we see
that the left side of \eqref{tf1} equals
\begin{equation*}
\begin{split}
     & \quad tr\,\big(H^2(G^*(1-P)[H,G]+G(1-P)[H,G^*])P\big)\\
     &\quad -tr\,\big(H([H,G^*](1-P)[H,G]+[H,G](1-P)[H,G^*])P\big)\\
     & =tr\,\big(HG^*H(1-P)[H,G]P+HGH(1-P)[H,G^*])P\big)\\
     & =tr\,\big(HPG^*(1-P)H(1-P)[H,G]P+HPG(1-P)H(1-P)[H,G^*])P\big)\\
     & =tr\,\big(HA^*H[H,A]+HBH[H,B^*]\big).\\
\end{split}
\end{equation*}
\end{proof}

Since commutators are not affected by replacing
$H$ by $H-z$, we have an immediate corollary:

\begin{corollary}\label{H-zCor}
 Under the assumptions of Theorem \ref{abstract-trace-general-theorem}, for all $z\in\mathbb{R}$:

\begin{equation}\label{tf2}
\begin{split}
   &
tr\,\big((z-H)^2(G^*[H,G]+G[H,G^*])P\big)\\
&\quad +tr\,\big((z-H)([H,G^*][H,G]+[H,G][H,G^*])P\big)\\
   &=\\
 & tr\,\big((z-H)A(z-H)^2A^*-(z-H)A^*(z-H)^2A\big)\\
 &\quad +tr\,\big((z-H)B^*(z-H)^2B-(z-H)B(z-H)^2B^*\big).\\
   \end{split}
\end{equation}

\end{corollary}

\begin{proofof}{Theorem \ref{AbstrTrace}}
We may write the first trace in Corollary \ref{H-zCor} in terms of
second commutators by applying the following algebraic identity,
which is a direct computation:

\begin{equation}\label{cf1}
G^*[H,G]+G[H,G^*]=\frac1{2}[G^*, [H,G]] + \frac1{2}[G, [H,G^*]]+\frac1{2}[H,GG^*+G^*G].
\end{equation}

When \eqref{cf1} is multiplied by $P$ and the trace is taken, the
last term vanishes, and for the left side of \eqref{tf1} we obtain

\begin{equation}\label{tf3}
tr\,\big(H^2(G^*[H,G]+G[H,G^*])P\big)=\frac1{2}tr\,\big(H^2([G^*, [H,G]]+[G, [H,G^*]])P\big).
\end{equation}

If the spectrum of $H$ consists only of eigenvalues $\lambda_j$, with an
orthonormal basis of eigenfunctions
$\left\{\phi_j\right\}$, then the trace identity \eqref{tf3} and Corollary \ref{H-zCor} imply
\begin{equation}\nonumber
\begin{split}
    &\;\frac1{2}\sum_{\lambda_j\in J}  (z - \lambda_j)^2\,\big(\langle[G^*,[H,G]]\phi_j,\phi_j\rangle+\langle[G, [H,G^*]]\phi_j,\phi_j\rangle\big)\\
    &-\sum_{\lambda_j\in
    J}(z-\lambda_j)\,\big(\langle[H,G]\phi_j,[H,G]\phi_j\rangle+\langle[H,G^*]\phi_j,[H,G^*]\phi_j\rangle\big)\\
    &=\\
    &\sum_{\lambda_j\in J}\sum_{\lambda_k\notin J}
    (z-\lambda_j)(z-\lambda_k)(\lambda_k-\lambda_j)\big(|\langle G\phi_j,\phi_k\rangle|^2+|\langle G^*\phi_j,\phi_k\rangle|^2\big),\\
    \end{split}
\end{equation}
establishing \eqref{tf2discrete}.
\end{proofof}

If $H$ has a gap in its spectrum, then we consider the spectral
projector $P$ that separates the two parts of the spectrum.
\begin{theorem}\label{gap-inequality-thm}
Let $G,H$ satisfy the assumptions of Theorem
\ref{abstract-trace-general-theorem}. Suppose there exist
constants $\lambda<\Lambda$ such that
\begin{equation}\label{gap-condition}
    HP\leq \lambda< \Lambda\leq  H(1-P).
\end{equation}
Then for all $z\in[\lambda,\Lambda]$, 
\begin{equation}\label{gap-inequality}
    \begin{split}
   &
tr\,\big((z-H)^2(G^*[H,G]+G[H,G^*])P\big)\\
&\quad +tr\,\big((z-H)([H,G^*][H,G]+[H,G][H,G^*])P\big)\\
   &\quad \quad \leq\\
 & tr\,\big((G^*[H,G]+G[H,G^*])P\big)(z-\lambda)(z-\Lambda).\\
   \end{split}
\end{equation}
\end{theorem}

\begin{remark}
While this formula only makes sense if the spectrum of
$H P$ is discrete, it is not necessary for the whole spectrum of
$H$ to be discrete.
\end{remark}

\begin{proofof}{Theorem \ref{gap-inequality-thm}}
We bound each term of the right side of \eqref{tf2}. Since
$z \in[\lambda,\Lambda]$ and $\text{Ran}\,(A^*A)\subseteq
\text{Ran}\,(P)$, with the cyclic property of traces
we get
\begin{equation*}
\begin{split}
    tr\,\big((z-H)A(z-H)^2A^*\big)&\leq
    (z-\Lambda)\;tr\,\big(A(z-H)^2A^*\big)\\
    &=(z-\Lambda)(z-\lambda)\;tr\,\big((z-H)^2A^*A\big)\\
    &\leq (z-\Lambda)(z-\lambda)\;tr\,\big((z-H)A^*A\big).
    \end{split}
\end{equation*}
Since $\text{Ran}\,(AA^*)\subseteq \text{Ran}\,(1-P)$, we obtain
similarly
\begin{equation*}
    -tr\,\big((z-H)A^*(z-H)^2A\big)\leq (z-\Lambda)(z-\lambda)\;tr\,\big((z-H)AA^*\big),
\end{equation*}
and therefore
\begin{equation*}
    tr\,\big((z-H)A(z-H)^2A^*-(z-H)A^*(z-H)^2A\big)\leq (z-\Lambda)(z-\lambda)\;tr\,\big((z-H)[A^*,A]\big).
\end{equation*}
In the same manner we estimate the second trace in \eqref{tf2}, obtaining
\begin{equation*}
\begin{split}
     &tr\,\big((z-H)A(z-H)^2A^*-(z-H)A^*(z-H)^2A\big)\\
     &\;+tr\,\big((z-H)B^*(z-H)^2B-(z-H)B(z-H)^2B^*\big)\\
     &\quad \quad \leq \\
     &(z-\Lambda)(z-\lambda)\;tr\,\big((z-H)[A^*,A]+(z-H)[B,B^*]\big)\\
     \end{split}
\end{equation*}

\noindent
Comparing the coefficients of $z^2$ in Corollary \ref{H-zCor}, we
see that
\begin{equation*}
    tr\,\big((G^*[H,G]+G[H,G^*])P\big)=tr\,\left((z-H)[A^*,A]+(z-H)[B,B^*]\right),
\end{equation*}
which proves the theorem.
\end{proofof}

If we take $P=P_{H<z} \,$, the spectral projector onto the
spectrum below $z$, then we can rewrite \eqref{gap-inequality} as
follows:
\begin{equation}\label{Riesz-mean-inequality}
    \begin{split}
   &
tr\,\big((z-H)_{+}^2(G^*[H,G]+G[H,G^*])\big)\\
&+tr\,\big((z-H)_{+}([H,G^*][H,G]+[H,G][H,G^*])\big) \leq 0,\\
   \end{split}
\end{equation}
where $(z - H)_+ := (z - H) P_{H<z}$. We extend this inequality to
the class of trace-controllable functions $f$ of \cite{HaSt2}: Let
$f$ be a $C^3$ function such that $f(0)=f'(0)=f''(0)=0$ and
$f'''(t)\geq 0$ for $t\geq 0$. From the identity
\begin{equation}\label{f}
    f(z-\lambda)=\frac1{2}\int_0^{\infty}(z-\lambda-t)_{+}^2f'''(t)\;dt
\end{equation}
we deduce the following inequality.
\begin{theorem}\label{f-gap-inequality-theorem} Let $G,H$ satisfy the assumptions of Theorem
\ref{abstract-trace-general-theorem} and let $f$ be as above. Then
\begin{equation}\label{f-gap-inequality}
    \begin{split}
   &
tr\,\big(f((z-H)_{+})(G^*[H,G]+G[H,G^*])\big)\\
&+\frac1{2}tr\,\big(f'((z-H)_{+})([H,G^*][H,G]+[H,G][H,G^*])\big) \leq 0.\\
   \end{split}
\end{equation}

\end{theorem}

\section{On the eigenvalues of periodic Schr\"odinger operators}\label{PeriodicSection}

In this section we suppose that $H$ is of the form

\begin{equation}\label{Hdef}
H = - \Delta + V({\bf x})
\end{equation}
and is defined as a self-adjoint operator on
$L^2(\Omega)$, where $\Omega \subset \R^d$ is a bounded domain
and the boundary conditions are such that the multiplication operator
$G = \exp(-i {\bf q} \cdot {\bf x})$ satisfies the domain-mapping conditions
of Theorem \ref{AbstrTrace}.  This situation arises in the Floquet decomposition of
$H$ when $V({\bf x})$ is a real, periodic, bounded measurable function
\cite{Kuc,OdKe,ReSiIV,Skr},
where $\Omega$ is a fundamental domain of periodicity and
${\bf q}$ is a vector of the reciprocal lattice.  It also covers the case of the
Dirichlet Laplacian, with the same $G$, the vector ${\bf q}$ being arbitrary.
Commutators with $\exp(-i {\bf q})$
are at the heart of the ``Bethe sum rule'' of quantum mechanics
\cite{BeJa,Wan} and
have appeared in some other analyses
of the distribution of eigenvalues in \cite{Her,Lap,LePa,Saf}, although
the specific
consequences for universal bounds for eigenvalues of
periodic operators have not, to our knowledge,
been explored before.

In Section \ref{LTSection}
we shall
introduce a semiclassical parameter $\alpha$ proportional to
the square of Planck's constant, and study
\begin{equation}\label{HSemi}
H_{\alpha} = - \alpha \Delta + V({\bf x}).
\end{equation}
Although it is always possible to reset $\alpha > 0$
to 1 by a change of scale, we introduce $H_{\alpha}$ in order to
study the semiclassical limit
$\alpha \to 0$.
A further possible extension would be to
introduce of a magnetic field through the systematic replacement
of $\nabla$ by $\nabla + i {\bf A}({\bf x})$; this entails only minor changes,
because in the key identities
the magnetic vector potential ${\bf A}({\bf x})$
occurs only in commutators
that vanish.  In the interest of clarity we leave this generalization as an exercise for the
interested reader.

The commutators appearing in Theorem \ref{AbstrTrace} are easily calculated:

\begin{equation}\label{1stcomm}
\left[H,G\right]=  \exp(-i {\bf q} \cdot {\bf x})\left(|{\bf q}|^2 + 2 i {\bf q} \cdot \nabla\right)
\end{equation}
and
\begin{equation}\label{GHG}
\left[G^*,\left[H,G\right]\right] = \left[G,\left[H,G^*\right]\right] = 2 |{\bf q}|^2.
\end{equation}
With these facts in hand, Theorem \ref{AbstrTrace} reads

\begin{equation}\nonumber
\begin{split}
    &\;2 |{\bf q}|^2 \sum_{\lambda_j\in J}  (z-\lambda_j)^2
    -\sum_{\lambda_j\in
    J}(z-\lambda_j)\,\left(2 |{\bf q}|^4 + 8 \|{\bf q} \cdot \nabla \phi_j\|^2 \right)\\
    &=\\
    &\sum_{\lambda_j\in J}\sum_{\lambda_k\notin J}
    (z-\lambda_j)(z-\lambda_k)(\lambda_k-\lambda_j)\big(|\langle G\phi_j,\phi_k\rangle|^2+|\langle G^*\phi_j,\phi_k\rangle|^2\big),\\
    \end{split}
\end{equation}
or
\begin{equation}\label{PerID}
\begin{split}
    &|{\bf q}|^2 \sum_{\lambda_j\in J}  (z-\lambda_j)^2
    -\sum_{\lambda_j\in
    J}(z-\lambda_j)\,\left(|{\bf q}|^4 + 4 \|{\bf q} \cdot \nabla \phi_j\|^2 \right)\\
    &=\\
    &\sum_{\lambda_j\in J}\sum_{\lambda_k\notin J}
    (z-\lambda_j)(z-\lambda_k)(\lambda_k-\lambda_j)w_{jk{\bf q}}\, ,\\
    \end{split}
\end{equation}

\noindent
where $w_{kj{\bf q}} := \frac{1}{2} \left(|\langle\exp(-i {\bf q} \cdot {\bf x})\phi_j,\phi_k\rangle|^2+|\langle \exp(i {\bf q} \cdot {\bf x})\phi_j,\phi_k\rangle|^2\right)$.
We collect here some properties of $w_{kj{\bf q}}$ and associated ``sum rules'' for the
eigenvalues:

\begin{proposition}\label{wprops}
The quantities $w_{kj{\bf q}}$ compose
an infinite doubly stochastic matrix, i.e.,

a)\quad\quad\quad\quad $\sum_k{w_{kj{\bf q}}} = \sum_j{w_{kj{\bf q}}} = 1$,

\noindent
with the symmetries

b)\quad\quad\quad\quad
$0 \le w_{kj{\bf q}} = w_{jk{\bf q}}$;
and

c)\quad\quad\quad\quad
$w_{kj-{\bf q}} = w_{kj{\bf q}}$.

\noindent
Moreover,
\end{proposition}

\begin{equation}\label{M1}
\sum_k{(\lambda_k - \lambda_j)w_{kj{\bf q}}} = |{\bf q}|^2,
\end{equation}
{\it and in particular, for each $j$,}
\begin{equation}\label{M1a}
\lambda_j = \sum_k{\lambda_k w_{kj{\bf q}}} - |{\bf q}|^2;
\end{equation}
\begin{equation}\label{M2}
\sum_k{(\lambda_k - \lambda_j)^2 w_{kj{\bf q}}} = |{\bf q}|^4 + 4 \|{\bf q} \cdot \nabla \phi_j\|^2 .
\end{equation}

\begin{proof}
Properties a)-c) are immediate from the definition of $w_{kj{\bf q}}$ and the
completeness relation of the eigenfunctions.

Choosing $J = \{\lambda_j\}$,
Identity \eqref{M1} results from
taking the second derivative of
\eqref{PerID} with respect to $z$.  Formula \eqref{M1a} is just a reformulation
of \eqref{M1}.  For \eqref{M2}, set $z = \lambda_k$, multiply \eqref{PerID}
by $w_{kj{\bf q}} \,$, and then sum on $j$.

\end{proof}

In the spirit of \cite{HaSt1,HaHe1}, we next exploit \eqref{PerID}
to obtain control over eigenvalues and their means.  For
$J = \{\lambda_j\}$, we find
$$(z-\lambda_j)^2-q_{\ell}^2 (z-\lambda_j)- 4 (z-\lambda_j)T_{{\bf q} j}
= H_{{\bf q} j},$$
where
$$T_{{\bf q} j}:=\frac{\|{\bf q}\cdot\nabla u_j\|^2}{q_{\ell}^2} \quad {\rm and} \quad
H_{{\bf q} j}:=\sum_k
{(z-\lambda_i)(z-\lambda_k)(\lambda_k-\lambda_j) \frac{w_{kj{\bf q}}}{q_{\ell}^2}}.
$$
For $J=\{\lambda_1,\dots,\lambda_N\}$, we sum in $j$,
defining
$$\overline{\lambda_N} := {\frac 1 N }\sum_{j
\le N}{\lambda_j}$$
and
$$
\overline{\lambda_N^2} := {\frac 1 N}
\sum_{j \le N}\lambda_{j}^2
$$
to write \eqref{PerID} as

\begin{equation}\label{R_2}
\sum_{j=1}^N(z-\lambda_j)^2 = N(z^2- 2 \overline{\lambda_N} \, z+
\overline{\lambda^2_N})
=q_{\ell}^2 \sum_{j=1}^N{(z-\lambda_j)} +
4 \sum_{j=1}^N{(z-\lambda_j)T_{{\bf q} j} } + H,
\end{equation}
where
$$H :=
\sum_{j=1}^N\sum_{k=N+1}^\infty
{(z-\lambda_j)(z-\lambda_k)(\lambda_k-\lambda_j)\frac{w_{kj{\bf q}}}{q_{\ell}^2}} \le 0$$
for $z\in [\lambda_N,\lambda_{N+1}]$.  (The contribution to the sum
for $k \le N$ has dropped out by symmetry.)

We wish next to let ${\bf q}$ range over the dual lattice and average;
in order to concentrate on the most
straightforward cases, we henceforth make two simplifying assumptions:

1.  The fundamental domain of periodicity $K$ of $V({\bf x})$ is rectangular, with
sides of length $2 \pi/ q_{\ell}$; and

2.  Periodic boundary conditions are imposed on functions in $L^2(K)$.  This incidentally allows us to choose the eigenfunctions $\{\phi_j\}$ to be real-valued, with consequent simplifications for
expressions such as $w_{kj{\bf q}}$.

\noindent
(If one considers the problem of a general Floquet multiplier for a periodic problem with a fundamental cell that is not rectangular,
then there are some complications of detail, but the results
remain similar to the ones presented here.)  With the simplifying assumptions the
basis of the dual lattice can be taken as consisting of multiples of the
Cartesian basis, and in \eqref{PerID} we may set
${\bf q}= q_{\ell} \hat e_\ell$, $\ell=1,\dots,d$.
With this choice $T_{{\bf q} j} = \|\frac{\partial\phi_j}{\partial x_\ell}\|^2$, and
\begin{equation}\label{KE}
\sum_{\ell = 1}^d{ T_{{\bf q} j}} =  \| \nabla \phi_j \|^2
= \lambda_j - \left\langle{\phi_j , V \phi_j }\right\rangle =: \lambda_j - V_j =: T_j.
\end{equation}
Let $g := \frac{1}{d} \sum_{\ell=1}^d{q_{\ell}^2}$, and define the ``Riesz means''
$$R_{\sigma}(z) := \sum_j{(z-\lambda_j)_+^{\sigma}}.$$
Then \eqref{R_2} can be summed on $\ell$, so that

\begin{equation}\label{difference1}
R_2(z) \le g R_1(z) + \frac{4}{d} \sum_j{(z - \lambda_j) T_j}.
\end{equation}
(In the case of periodicity with respect to basis vectors
$\{{\bf q}_{\ell}\}$ that are not orthogonal,
the factor $T_j$ may be replaced by $C T_j$ for a constant $C \ge 1$
determined by
the geometry of the set $\{{\bf q}_{\ell}\}$.)
Because of \eqref{KE}, this inequality is equivalent to the statement that
a certain quadratic polynomial, {\it viz.},

\begin{equation}\label{quadpoly2}
z^2 - \left(\left(2 + \frac 4 d \right) \overline{\lambda_N} + g -
\frac 4 d \overline{V_N}\right) z +
\left(1 + \frac 4 d \right) \overline{\lambda_N^2} + g \overline{\lambda_N} -
\frac 4 d \overline{\lambda V_N} \le 0
\end{equation}
for all $z \in \left[\lambda_N, \lambda_{N+1}\right]$.
Here, in keeping with the notation for the averages of
eigenvalues and their squares,
$\overline{V_N} := {\frac 1 N }\sum_{j
\le N}{V_j}$ and $\overline{\lambda V_N} := {\frac 1 N }\sum_{j
\le N}{\lambda_j V_j}$.  Letting $z = \lambda_{N+1}$, we obtain
a universal inequality on $\lambda_{N+1}$ to be compared with
a similar result for the Dirichlet Laplacian in \cite{Yan}, cf. also
Proposition 6 of \cite{HaSt1}:

\begin{equation}\label{upperbd}
\lambda_{N+1} \le
\left(1 + \frac 2 d \right) \overline{\lambda_N} + \frac{g - \overline{V_N}}{2} +
\sqrt{{\mathfrak D}_N},
\end{equation}
where ${\mathfrak D}_N$, defined as
the discriminant of the quadratic polynomial in \eqref{quadpoly2},
is guaranteed to be $\ge 0$.

%

For the problem of Schr\"odinger operators on bounded domains with Dirichlet
conditions we may let $g \to 0$, and if moreover $V = 0$, then \eqref{difference1}
reduces to the inequality of Yang for the Dirichlet Laplacian \cite{Yan,HaSt1,Ash}.

\section{Universal monotonicity of Riesz means for periodic Schr\"odinger
operators and Schr\"odinger operators on manifolds of bounded mean
curvature}\label{UniversalRiesz}

Inequality \eqref{difference1} is identical in form to a bound that
applies for a suitable value of
$g$ to Schr\"odinger operators on
immersed manifolds of dimension $d$,
according to Corollary 4.3 of \cite{Har},
which was proved there with
a different commutator argument.
 (In \cite{Har} refer to (1.10) and (4.2) to elucidate the notation,
and note that the $n$'s in
the denominators in Corollary 4.3 are incorrect and should be
deleted.)
The inequality of \cite{Har}
can equally well be
proved with the methods of this article, as will be shown in
Section \ref{UnitSection}.
It is therefore possible to derive monotonicity properties and
eigenvalue bounds simultaneously for these two 
categories of Schr\"odinger operator.  In the earlier article, the additional
term comes from the mean curvature of an immersed hypersurface, and thus reflects the way in
which the manifold can be immersed in Euclidean space.  The periodic problem could be similarly regarded as about Schr\"odinger operator on a flat torus, but if instead of making arguments based on the fundamental domain one embedded the torus in $\R^{d+1}$, then one would
gain an additional term in the trace inequality from the associated mean curvature rather than from
the basis of the dual lattice, and would effectively convert the situation of Section \ref{PeriodicSection}
into that of
\cite{Har}.  We refer to \cite{EHI} for extensions of \cite{Har} in various directions of geometric
interest.

We shall refer to an operator $H = - \Delta + V({\bf x})$ or
$H_{\alpha}=- \alpha \Delta + V({\bf x})$ as a {\it Schr\"odinger
operator on a manifold $M$ of bounded mean curvature} when
$\Omega\subset M$ is a domain in a smooth closed manifold $M$
immersed with finite mean curvature $h :=
\sum_{\ell=1}^d{\kappa_{\ell}}$ in $\R^{d+1}$, Dirichlet
conditions being imposed on $\partial \Omega$ if it is nonempty,
and the potential $V({\bf x})$ is a real, periodic, bounded
measurable function. Setting $\alpha=1$, by Corollary 4.3 of
\cite{Har} we get \eqref{difference1} with $g =
\frac{\sup(h)^2}{d}$. (An independent proof of this is given
below, cf. \eqref{GeomIneq}.)

As in \cite{HaHe1},
because $R_2^{\prime}(z) = 2 R_1(z)$,
the difference inequality \eqref{difference1} for the periodic problem is equivalent to a differential inequality
in the variable $z$:

\begin{equation}\label{differential1}
\left(z + \frac{g d}{4} \right) R_2^{\prime}(z) \ge \left(2 + \frac{d}{2} \right) R_2(z) +
2 \sum_j {(z-\lambda_j)_+V_j}.
\end{equation}

Simplifying by replacing $V_j$ by $\sup V$ and letting $\tau := g d/4 - \sup V$, we
obtain the inequality

\begin{equation}\label{differential2}
\left(z + \tau \right) R_2^{\prime}(z) \ge \left(2 + \frac{d}{2} \right) R_2(z) .
\end{equation}

A Schr\"odinger operator $H$ on a manifold of bounded mean
curvature similarly satisfies \eqref{differential2} with $\tau :=
\frac{\sup(h)^2}{4} - \sup V$. The differential inequality
\eqref{differential2} is easily solved, and we have thus proved:

\begin{theorem}\label{monotone1}
Let $H=-\Delta + V({\bf x})$ be a periodic Schr\"odinger operator
with fundamental domain $M$ or a Schr\"odinger operator on a
bounded manifold $M$ of bounded mean curvature and finite volume.
Then the function
$$
\frac{R_2(z)}{(z+\tau)^{2+d/2}},
$$
with $\tau = \frac{g d}{4} - \sup V$ or respectively $\frac{\sup(h)^2}{4} - \sup V$, 
is nondecreasing for
all $z$ real.  Consequently,

$$
\frac{R_2(z)}{(z+\tau)^{2+d/2}} \le L_{2,d}^{\rm cl} {\rm Vol}(M).
$$
where
$L_{\sigma,d}^{cl}=(4\pi)^{-\frac{d}{2}}\frac{\Gamma(3)}{\Gamma(3+\frac{d}{2})}$.
\end{theorem}

\begin{corollary}\label{Leg1}
For $k \ge j$, the means of the eigenvalues of
$H$ satisfy

\begin{equation}\label{LegP8}
    \frac{d+2}{d}\;(\overline{\lambda_k} + \tau)
    \leq
    \bigg(\frac{k}{n}\bigg)^{\frac{2}{d}}\left(\frac{d+2}{d}\;(\overline{\lambda_n} + \tau)+\sqrt{D_n}\right),
\end{equation}
where
\begin{equation}\label{Disc}
D_n := \left(1 + \frac 2 d \right)^2 (\overline{\lambda_n} + \tau)^2 -
\left(1 + \frac 4 d \right) \left(\overline{\lambda_n}^2 + 2 \overline{\lambda_n} \tau
+ \tau^2\right).
\end{equation}
Consequently,
\begin{equation}\label{LegP9}
\frac{\overline{\lambda_{k}} + \tau}{\overline{\lambda_{j}} + \tau} \le
\frac{d + 4}{d + 2}\left( {\frac k j}\right)^{\frac 2 d}.
\end{equation}
\end{corollary}

\begin{remarks}\label{BetterHarHer}

\end{remarks}

a)
The more appealing bound \eqref{LegP9} is strictly weaker than
\eqref{LegP8}, which has the virtue of being sharp in the Weyl
limit.

b)
The corollary applies in particular to the Dirichlet Laplacian
on a bounded domain, with $\tau = 0$.  In this case it improves
a recent inequality of \cite{HaHe2} both in terms of the constant and in
the range of $j,k$ for which it is valid.

c)
A similar monotonicity theorem can be proved for
$$
\frac{R_{\sigma}(z)}{(z+\tau)^{\sigma+d/2}}
$$
with $\sigma > 2$ as in \cite{HaHe1}).

\begin{proof}
It suffices to prove the corollary
assuming that $\tau = 0$, as the effect on the eigenvalues
of adding $\tau$ to
$z$ is
equivalent to a systematic shift of $\tau$ in each eigenvalue.
For any positive integer $n$ we consider the function
$P_{2,n}:[\lambda_n,\infty) \rightarrow \mathbb{R}_{+}$
defined by
\begin{equation}\label{P-sigma-n}
    P_{2,n}(z):=
    \sum_{j=1}^{n}(z-\lambda_j)\left(z-\left(1+\frac{4}{d}\right)\lambda_j\right).
\end{equation}
From \eqref{upperbd} we can see that
$ P_{2,n}(z)\leq 0$ for all
$z\in (\lambda_n,\lambda_{n+1})$.
As a consequence $P_{2,n}(z)$ has a
largest zero $z^0_n \geq \lambda_n$. Since for all $z\geq
\lambda_n$,
\begin{equation*}
    R_{2}(z)\geq \sum_{j=1}^n(z-\lambda_j)^{2},
\end{equation*}
and as in  \cite{HaHe1} we conclude from \eqref{differential2} that
for all $\zeta\geq z \geq \lambda_n$,
\begin{equation}\label{R-sigma-ineq}
    \zeta^{-2-\frac{d}{2}}R_{2}(\zeta)\geq
    z^{-2-\frac{d}{2}}\sum_{j=1}^n(z-\lambda_j)^{2}.
\end{equation}
We want to optimize the
right side of \eqref{R-sigma-ineq} with respect to $z$.
Since
\begin{equation}\label{R-P-relation}
    \frac{d}{dz}\;z^{-2-\frac{d}{2}}\sum_{j=1}^n(z-\lambda_j)^{2}=-\frac{d}{2}\;P_{2,n}(z)z^{-2-\frac{d}{2}-1},
\end{equation}
an optimal choice is $z^0_n$, where
\begin{equation}\label{z0n}
    z^0_n=\frac{d+2}{d}\overline{\lambda_n}+\sqrt{D_n} \le \frac{d+4}{d}\overline{\lambda_n},
\end{equation}
in which
$D_n$ is the discriminant of the quadratic.  (See \cite{HaSt1} for further details.)
Hence for all
$\zeta\geq z^0_n$,
\begin{equation*}
    \zeta^{-2-\frac{d}{2}}R_{2}(\zeta)\geq (z^0_n)^{-2-\frac{d}{2}}\sum_{j=1}^n(z^0_n-\lambda_j)^{2}
    = \frac{(z^0_n)^{1-2-\frac{d}{2}}}{1+\frac{d}{4}}\sum_{j=1}^n(z^0_n-\lambda_j).
\end{equation*}
Since $R'_{2}(\zeta)=2 R_1(\zeta)$, it follows
from \eqref{differential2} that
\begin{equation*}
    R_{2}(\zeta)\leq
    \frac{\zeta}{1+\frac{d}{4}}R_1(\zeta).
\end{equation*}
Consequently, for all $\zeta\geq z^0_n$,
\begin{equation}\label{R-sigma-main-ineq}
    \zeta^{-1-\frac{d}{2}}R_1(\zeta)\geq  (z^0_n)^{-1-\frac{d}{2}}\sum_{j=1}^n(z^0_n -\lambda_j).
\end{equation}
Since $z^0_n-\overline{\lambda_n} \geq\frac{2}{d+2}z^0_n$,
\begin{equation}\label{R-1-main-ineq}
    \zeta^{-1-\frac{d}{2}}R_{1}(\zeta)\geq  \frac{n}{1+\frac{d}{2}}\;(z^0_n)^{-\frac{d}{2}}.
\end{equation}
We note parenthetically that
estimate \eqref{R-1-main-ineq} is asymptotically sharp, since
\begin{equation*}
    \underset{\zeta\rightarrow\infty}{\lim}\zeta^{-1-\frac{d}{2}}R_{1}(\zeta)=\frac{|\Omega|}{1+\frac{d}{2}}\; C_d^{-\frac{d}{2}}=
     \underset{n\rightarrow\infty}{\lim}\frac{n}{1+\frac{d}{2}}\;(z^0_n)^{-\frac{d}{2}},
\end{equation*}
where $C_d$ denotes the classical constant given by the Weyl limit,
\begin{equation*}
    C_d=\underset{n\rightarrow\infty}{\lim}\lambda_n(\frac{n}{|\Omega|})^{-\frac{2}{d}}
    = (2 \pi)^2 {\rm Vol}(S^{d})^-{2/d}.
\end{equation*}
We now rewrite \eqref{R-1-main-ineq} as
\begin{equation*}
    R_{1}(\zeta)\geq  \frac{n}{1+\frac{d}{2}}\;(z^0_n)^{-\frac{d}{2}}\zeta^{1\frac{d}{2}}
\end{equation*}
and take the Legendre transformation on both sides, following
standard calculations
to be found, e.g., in \cite{LaWe,HaHe1}.  The result is that if
$w$ is restricted to values $\geq n$, then
\begin{equation}\label{R-1-LegendreTrans-ineq}
    (w-[w])\lambda_{[w]+1}+\sum_{k=1}^{[w]}\lambda_k\leq \frac{z^0_n}{(1+\frac{2}{d})n^{\frac{2}{d}}}\;w^{1+\frac{2}{d}}.
\end{equation}
Hence for all $k\geq n$ (letting $w$ approach $k$ from below) we
get
\begin{equation}
    \frac{d+2}{d}\;\overline{\lambda_k}\leq \bigg(\frac{k}{n}\bigg)^{\frac{2}{d}}z^0_{2
    ,n}=\bigg(\frac{k}{n}\bigg)^{\frac{2}{d}}\left(\frac{d+2}{d}\;\overline{\lambda_n}+\sqrt{D_n}\right),
\end{equation}
which proves the theorem.  (The simplification
\eqref{LegP9} is achieved with the upper bound in \eqref{z0n}.)

\end{proof}

\section{Universal monotonicity of eigenvalue moments and sharp Lieb-Thirring inequalities for periodic Schr\"odinger
operators and Schr\"odinger operators on manifolds of bounded
mean curvature}\label{LTSection}

We next turn our attention to the one-parameter family of operators
$H_{\alpha}$ from \eqref{HSemi} in order to derive inequalities of
Lieb-Thirring type
for periodic Schr\"odinger operators and for Schr\"odinger
operators on domains of bounded mean curvature. Some inequalities of
Lieb-Thirring type appear in \cite{Il1,Il2,Il3} for Schr\"odinger
operators on spheres after projection onto 
the set of functions of mean zero,
and Sobolev type inequalities related to Lieb-Thirring may be
found in \cite{Tem}.  We shall use the direct method introduced in
\cite{Stu} to derive an explicit form of a Lieb-Thirring
inequality for eigenvalue moments of order $\sigma\geq 2$, without
projection.

For the purposes of semiclassical analysis, we appeal to the
Feynman-Hellman theorem to note that
\begin{equation}\label{FeynHel}
T_j = \frac{\partial \lambda_j}{\partial \alpha}
\end{equation}
(except at eigenvalue crossings, cf. \cite{Stu}), and therefore,
after scaling to incorporate
$\alpha$ and introducing an integrating factor,
\eqref{difference1} reads

\begin{equation}\label{scaleddifferential}
\frac{\partial}{\partial \alpha} \left(\alpha^{d/2} R_2(z, \alpha)\right) \le
\frac{g d}{2} \alpha^{d/2} R_1(z, \alpha).
\end{equation}

Recalling that $\partial R_2/\partial z = 2 R_1$, we see that
\eqref{scaleddifferential} can be regarded as a partial differential inequality.
Letting $U(z, \alpha) := \alpha^{d/2} R_2(z, \alpha)$, the inequality has the form

\begin{equation}\label{PDI}
\frac{\partial U}{\partial \alpha}  \le \frac{g d}{4} \frac{\partial U}{\partial z},
\end{equation}
which can be solved by changing to characteristic variables
$\xi := \alpha - \frac{4}{g d} z$,
$\eta := \alpha + \frac{4}{g d} z$, in terms of which
$$
\frac{\partial U}{\partial \xi}  \le 0,
$$
i.e., $U$ decreases as $\xi$ increases while $\eta$ is fixed.  In conclusion,

\begin{equation}\label{UIneq}
U(\alpha, z)  \le U\left(\alpha_s, z+ \frac{g d}{4} (\alpha - \alpha_s)\right)
\end{equation}
for $\alpha \ge \alpha_s$.

As $\alpha_s \to 0$, the right side of \eqref{UIneq} tends to
$L_{2,d}^{cl} \int{\left|V({\bf x}) - \left(z + \frac{g d}{4}
\alpha\right)\right|_-^{2 + d/2} d{\bf x} }$.
Since (see e.g. \cite{BiSo}, \cite{BlSt}, \cite{Hu}, \cite{RoSo}
and references therein)
for all $\sigma\geq 0$,
\begin{equation}\label{sc-limit}
    \underset{\alpha\rightarrow 0+}{\lim}\alpha^{\frac{d}{2}}\;
    \sum_{E_j(\alpha)<z}(z-E_j(\alpha))^{\sigma}=L_{\sigma,d}^{cl}\int_M{\left(V({\bf x}) -
z\right)_{-}^{\sigma + d/2} d{\bf x}},
\end{equation}
with $L_{\sigma,d}^{cl}$, the {\it classical constant}, given by
\begin{equation}\label{sc-constant}
    L_{\sigma,d}^{cl}=(4\pi)^{-\frac{d}{2}}\frac{\Gamma(\sigma+1)}{\Gamma(\sigma+\frac{d}{2}+1)},
\end{equation}
we arrive at a sharp Lieb-Thirring inequality  for $R_2$:

\begin{theorem}\label{LT}
For all $\alpha>0$ the mapping
\begin{equation}\label{R-mononoticity-periodic}
    \alpha\mapsto {\alpha}^{\frac{d}{2}}R_2(z-\frac{\alpha gd}{4})={\alpha}^{\frac{d}{2}}\sum
    (z-\frac{\alpha gd}{4}-\lambda_j)_{+}^2
\end{equation}
is nonincreasing and therefore for all $z\in\mathbb{R}$ and all
$\alpha>0$ the following sharp Lieb-Thirring inequality holds:
\begin{equation}\label{LTPeriodic}
R_2(z, \alpha)
\le \alpha^{-d/2} L_{2,d}^{cl} \int_M{\left(V({\bf x}) - \left(z +
\frac{gd}{4} \alpha\right)\right)_{-}^{2 + d/2}d{\bf x}}.
\end{equation}

\end{theorem}
A similar monotonicity property can be proved for
$R_{\sigma}(z,\alpha)$ with $\sigma > 2$  (see also \cite{Stu}).
Indeed:
\begin{corollary}
For all $\alpha>0$ the mapping
\begin{equation}\label{R-mononoticity-periodic-sigma}
    \alpha\mapsto {\alpha}^{\frac{d}{2}}R_{\sigma}(z-\frac{\alpha gd}{4})={\alpha}^{\frac{d}{2}}\sum
    (z-\frac{\alpha gd}{4}-\lambda_j)_{+}^{\sigma}
\end{equation}
is nonincreasing and therefore for all $z\in\mathbb{R}$ and all
$\alpha>0$ the following sharp Lieb-Thirring inequality holds:
\begin{equation}\label{LTPeriodic-sigma}
R_{\sigma}(z, \alpha)
\le \alpha^{-d/2} L_{\sigma,d}^{cl} \int_M{\left(V({\bf x}) -
\left(z + \frac{gd}{4} \alpha\right)\right)_{-}^{\sigma +
d/2}d{\bf x}}.
\end{equation}
\end{corollary}

The conclusion of Theorem \ref{LT} also holds in the presence of
vector potentials. In particular, in the periodic case the
operator $\alpha(-i\mathbf{\nabla}+\mathbf{k})^2+V(\mathbf{x})$,
with $\mathbf{k}\in\mathbb{R}^d$ (more precisely, $\mathbf{k}$ in
the dual lattice) being a constant vector, satisfies the
Lieb-Thirring inequality \eqref{LTPeriodic}. Therefore, taking the
average over a band, which we define by
\begin{equation*}
    \langle \lambda_j\rangle :=\frac1{{\rm Vol}(M)}\int
    \lambda_j(\mathbf{k})\;d\mathbf{k},
\end{equation*}
and using the convexity of the function $\lambda\mapsto
(z-\lambda)_{+}^{\sigma}$ we get the estimate

\begin{equation}\label{LT-average}
   \sum
    (z-\langle \lambda_j\rangle)_{+}^{\sigma} \le \alpha^{-d/2} L_{\sigma,d}^{cl} \int_M{\left(V({\bf x}) -
\left(z + \frac{gd}{4} \alpha\right)\right)_{-}^{\sigma +
d/2}d{\bf x}}.
\end{equation}

We close the section with an application to a Schr\"{o}dinger operator with
a perturbation of a periodic potential with known semiclassical asymptotics
\cite{BiSo}. On $\mathbb{R}^d$ we consider the periodic
Schr\"{o}dinger operator with a fundamental domain $M\subset
\mathbb{R}^d$,
\begin{equation}\label{H-0-periodic}
    H_0=-\Delta +w(x),
\end{equation}
$w$ being a bounded measurable periodic function, $|w(x)|\leq w_0 < \infty$
As in \cite{BiSo} we suppose that $H_0\geq 0$ with $0$ being the
greatest lower bound. The foregoing argument yields a Lieb-Thirring inequality
for the Schr\"{o}dinger operator
\begin{equation}\label{H-0-periodic2}
    H(\alpha):=\alpha H_0+V(x).
\end{equation}
on $\mathbb{R}^d$, where $V$ is a continuous function of compact
support.

\begin{theorem}\label{LT-2}
For all $\alpha>0$ the mapping
\begin{equation}\label{R-mononoticity-periodic-2}
    \alpha\mapsto {\alpha}^{\frac{d}{2}}R_2(z-\alpha w_0)={\alpha}^{\frac{d}{2}}\sum
    (z-\alpha w_0-\lambda_j)_{+}^2
\end{equation}
is nonincreasing, and therefore for all $z\in\mathbb{R}$ and all
$\alpha>0$, the following sharp Lieb-Thirring inequality holds:
\begin{equation}\label{LTPeriodic-2}
R_2(z, \alpha)
\le \alpha^{-d/2} L_{2,d}^{cl} \int{\left(V({\bf x}) - \left(z
+\alpha w_0\right)\right)_{-}^{2 + d/2}d{\bf x}}.
\end{equation}

\end{theorem}

\section{Remarks on the commutation of self-adjoint and unitary operators}
\label{UnitSection}

In Section \ref{PeriodicSection} the operator $G$ that was chosen to commute with
the self-adjoint operator $H$ was unitary, and it is reasonable to think that
this property alone accounts for some of the simplifications that were achieved
in comparison with the general trace inequalities of Section \ref{IDSection}.
In this section we
choose $G = U$ as a unitary operator and explore some consequences and
additional applications.

We define
\begin{equation}\label{H_U}
H_U := U^*[H,U] = U^*HU - H.
\end{equation}
Then $HU  = U[H,U^*]$, and we may rewrite the trace formula
\eqref{tf2} as follows:
\begin{align}\label{U.2}
&  \quad tr\left((z-H)^2(H_U + H_{U^*})P\right) - tr\left((z-H)(H_U^2 + HUH_{U^*}^2)P\right) \nonumber
\\
& \quad  \quad = \nonumber \\
&  \quad tr\left((z-H)A(z-H)^2A^* - (z-H)A^* (z-H)^2 A\right)  + \nonumber \\
& \quad  \quad tr\left((z-H)B^*(z-H)^2 B - (z-H) B (z-H)^2B^*\right).
\end{align}

As before, if the spectrum of
$H$ consists only of eigenvalues $\lambda_j$ and the corresponding
eigenfunctions $\phi_j$ are chosen to form an orthonormal basis
of the underlying Hilbert space ${\mathcal H}$,
then \eqref{U.2} reads as follows:

\begin{align}\label{U.3}
&  \quad \sum_{\lambda_j \in J}{(z-\lambda_j)^2
\left\langle(H_U + H_{U^*})\phi_j, \phi_j\right\rangle} - \nonumber \\
& \quad \quad \sum_{\lambda_j \in J}{(z-\lambda_j)\left(\left\langle(H_U \phi_j, H_U \phi_j\right\rangle
+ \left\langle(H_{U^*} \phi_j, H_{U^*} \phi_j\right\rangle \right)} \nonumber
\\
& \quad  \quad = \nonumber \\
&  \quad \sum_{\lambda_j \in J}{\sum_{\lambda_k \notin J}{(z-\lambda_j)(z-\lambda_k)
(\lambda_k-\lambda_j)\left(\left|\left\langle(U \phi_j, \phi_k\right\rangle\right|^2+\left|\left\langle(U^* \phi_j, \phi_k\right\rangle\right|^2\right)}}.
\end{align}

Equivalently, this may be written as
\begin{align}\label{U.4}
&  \quad \sum_{\lambda_j \in J}{\|(z-H) \phi_j\|^2
\left\langle((2 z - U H U^* - u^* H U )\phi_j, \phi_j\right\rangle} - \nonumber \\
& \quad \quad \sum_{\lambda_j \in J}{\left\langle (z-H) \phi_j, \phi_j \right\rangle
\left(\|(z- U H U^*) \phi_j\|^2
+ \|(z- U^* H U) \phi_j\|^2 \right)} \nonumber
\\
& \quad  \quad = \nonumber \\
&  \quad \sum_{\lambda_j \in J}{\sum_{\lambda_k \notin J}{(z-\lambda_j)(z-\lambda_k)
(\lambda_k-\lambda_j)\left(\left|\left\langle(U \phi_j, \phi_k\right\rangle\right|^2+\left|\left\langle(U^* \phi_j, \phi_k\right\rangle\right|^2\right)}}.
\end{align}

In Section \ref{PeriodicSection}  we used the auxiliary unitary operator
$U = e^{-i\mathbf{q}\cdot\mathbf{x}}$ to derive identities for
for periodic Schr\"odinger operators.  As another illustration
we turn to the case of Schr\"odinger operators on manifolds
immersed in $\R^{d+1}$, which was studied by
commutation with self-adjoint operators
based on coordinate functions in \cite{Har,EHI}.

Consider a Schr\"odinger operator $H = - \Delta + V({\bf x})$, where
$-\Delta$ denotes the Laplace-Beltrami operator, on
a domain
$\Omega$ in a smooth, orientable, $d$-dimensional
manifold ${\mathcal M}$
isometrically immersed in $\R^{d+1}$.  (Higher codimensions could also
be dealt with as in \cite{EHI}, but
for simplicity we treat only the case of codimension 1.)  If
$\Omega$
has a boundary, Dirichlet conditions are imposed.  It is assumed that the potential
$V \in L_{\rm loc}^1$, and other conditions are tacitly placed on $V$ and
$\Omega$ so that $H$ is self-adjoint by closure of $C_c^\infty({\mathcal M})$ with at least some discrete,
finitely degenerate eigenvalues $\{\lambda_j\}$ at the bottom of the spectrum.
In order to apply a trace identity we choose $U$ as the multiplicative
operator obtained by
restricting $e^{-i\mathbf{q}\cdot\mathbf{x}}$ to values of $\mathbf{x} \in {\mathcal M}
\subset \R^{d+1}$.  We then calculate:

\begin{equation}\label{firstgeomcomm}
H_U = - 2 i {\bf q}_{\|} \cdot \nabla_{\|} - i {\bf q}\cdot{\bf h} + |{\bf q}_{\|}|^2.
\end{equation}
Here, ${\bf q}_{\|}$ and $\nabla_{\|}$ are the tangential parts of
${\bf q}$ and the gradient, while
the mean-curvature vector
${\bf h} = \left(\sum_{\beta}^{}{\kappa_{\beta}}\right) {\bf n}$
is parallel to the unit normal
${\bf n}$.
Using \eqref{U.3}, the analogue of \eqref{PerID} is

\begin{equation}
\begin{split}
    &|{\bf q}_{\|}|^2 \sum_{\lambda_j\in J}  (z-\lambda_j)^2
    -\sum_{\lambda_j\in
    J}(z-\lambda_j)\,\left(|{\bf q}_{\|}|^4 + 4 \|{\bf q}_{\|} \cdot \nabla_{\|} \phi_j\|^2+ \|{\bf q} \cdot {\bf h}\, \phi_j\|^2 \right)\\
    &\quad\quad =\\
    &\sum_{\lambda_j\in J}\sum_{\lambda_k\notin J}
    (z-\lambda_j)(z-\lambda_k)(\lambda_k-\lambda_j)w_{jk{\bf q}}\\
    \end{split}
\end{equation}
for some positive quantities ${w}_{jk{\bf q}}$ whose formal
expression is identical to the ones in
\eqref{PerID}.  To simplify this expression, one
can sum as before for ${\bf q}$ taken
from a frame of the form $\{q \, {\bf e}_{\beta}\}$, $\beta = 1, \dots d+1$.
(The same conclusion could alternatively be attained by fixing $|{\bf q}|$ and
averaging over all directions.)  The result, after a bit of calculation, is

\begin{equation}\label{GeomID}
\begin{split}
   & \sum_{\lambda_j\in J}{(z - \lambda_j)^2} - q^2 \sum_{\lambda_j\in J}{(z - \lambda_j)} - \frac{4}{d} \sum_{\lambda_j\in J}{(z - \lambda_j)
   \left({ \left\langle{\phi_j, \left(- \Delta + \frac{h^2}{4} \right) \phi_j}\right\rangle}\right)}\\
    & \quad \quad=\\
    &\sum_{\lambda_j\in J}\sum_{\lambda_k\notin J}
    (z-\lambda_j)(z-\lambda_k)(\lambda_k-\lambda_j)\frac{{\rm Ave}_{{\bf q}}\left(w_{jk{\bf q}}\right)}{q^2 d}.\\
    \end{split}
\end{equation}
In this identity it was convenient to assume a purely discrete spectrum,
although in fact only $J$ needs to be a discrete set, if the sum over $J^c$ is replaced
by the appropriate spectral integral.

The situation of greatest interest is
when $J = \{\lambda_1, \dots \lambda_n\} < \inf{J^c}$,
in which case the term on the right is nonpositive, and we may let $q \to 0$, yielding

\begin{equation}\label{GeomIneq}
R_2(z) \le \frac{4}{d} \sum_{\lambda_j\in J}{(z - \lambda_j)
   \left({ \left\langle{\phi_j, \left(- \Delta + \frac{h^2}{4} \right) \phi_j}\right\rangle}\right)}
\end{equation}
for all $z$, or, in the case where the spectrum is not purely discrete,
$z \le \inf{\sigma_{\rm ess}}$.
As remarked in Section \ref{UniversalRiesz}, this implies Inequality \eqref{differential1} and thereby
Theorem \ref{monotone1}, again in the case of purely discrete spectrum.
We observe that for the monotonicity part of
Theorem \ref{monotone1} 
it is not necessary for the manifold $M$ to be bounded or of finite volume, as long as the bottom part of the spectrum is discrete and $z$ lies below that.

A basic question about the spectral geometry of immersed manifolds has to do with
extensions of the Reilly inequality \cite{Rei,ElIl86,ElIl92,ElIl00,EHI},
whereby the eigenvalues of the Laplace-Beltrami operator are
bounded from above in terms of mean curvature.  The classic Reilly Inequality
says that the first non-trivial eigenvalue, which in our notation is $\lambda_2$,
is bounded by $\frac{\|h\|_{\infty}^2}{d}$.
In \cite{EHI}, Corollary 2.3,
it was shown that
every eigenvalue
$\lambda_N$ of the Laplace-Beltrami operator on closed immersed manifolds
satisfies an upper bound of the form $C_R(d,N) \|h\|_{\infty}^2$, but
unfortunately the constant $C_R(d,N)$ produced there
grows exponentially with $N$.

For $N>1$ the Reilly bound
on $\lambda_{N+1}$ can be improved in a form that grows as $N^{2/d}$, the
power expected from the Weyl law:

\begin{theorem}\label{Reilly}
Let a smooth, compact, $d$-dimensional
manifold ${\mathcal M}$, of finite volume and without boundary, be
immersed in $\R^{d+1}$.  Let $0 = \lambda_1 < \lambda_2 \le \dots$
denote the eigenvalues of the Laplace-Beltrami operator
on ${\mathcal M}$.  Then for
each $N$,
\begin{equation}\label{ReillyIneq}
\lambda_{N+1} \le \left( \frac{(d+4)^2}{d (d+2)} N^{2/d} - \frac 4 d \right)
        \frac{\|h\|_{\infty}^2}{d}.
\end{equation}
\end{theorem}

\begin{proof}
We start with \eqref{quadpoly2} as adapted to this situation, {\it viz.},
\begin{equation}
z^2 - \left(\left(2 + \frac 4 d \right) \overline{\lambda_N} + \tau \right) z +
\left(1 + \frac 4 d \right) \overline{\lambda_N^2} + \tau \overline{\lambda_N} \le 0,
\end{equation}
where $\tau =  \frac{\|h\|_{\infty}^2}{d}$ and $z \in (\lambda_k, \lambda_{k+1}]$,
and the corresponding specialization of the upper bound \eqref{upperbd}.
In a standard fashion we use Cauchy-Schwarz to replace
$\overline{\lambda_k^2} \ge \overline{\lambda_k}^2$
and weaken \eqref{upperbd} to:
\begin{equation}\label{weakerupperbd}
\lambda_{N+1} \le
\left(1 + \frac 4 d \right) \overline{\lambda_N} + \tau.
\end{equation}
As was already noted in \cite{EHI}, the case $N=1$ reproduces the classic Reilly inequality.  In order to bound higher eigenvalues, we now combine
\eqref{weakerupperbd} with \eqref{LegP9}, choosing $k=N$ and
$j=1$.  Recalling that $\lambda_1 = 0$,

\begin{equation}
\overline{\lambda_{N}} \le   \left( \frac{d + 4}{d + 2} N^{\frac 2 d}  - 1\right) \tau.
\end{equation}
When this is substituted into \eqref{weakerupperbd}, we obtain \eqref{ReillyIneq}.
\end{proof}

As a final application of commutation with unitaries,
consider the integral operator $H$
defined on $L^2(\mathbb{R}^d)$ by
\begin{equation}\label{H-integral-operator}
    (Hf)(p)=T(p)f(p)+\int_{\mathbb{R}^d}V(p-p')f(p')\;dp',
\end{equation}
and let $U$ be the translation operator $(Uf)(p)=f(p-k)$ for some
$k\in \mathbb{R}^d$.
This is in a sense dual to the situation given above, as
the unitary operator of multiplication by
$e^{-i\mathbf{q}\cdot\mathbf{x}}$
corresponds to translation in the momenta by $\mathbf{q}$.
Then
\begin{equation*}
    (H_Uf)(p)=(T(p+k)-T(p))f(p).
\end{equation*}
For simplicity we assume that the spectrum of $H$ consists only of
eigenvalues. Applying the trace formula \eqref{U.4} we
get
\begin{equation}\label{tf-U-integral-operator}
    \begin{split}
    &\sum_{\lambda_j\in J}  (z-\lambda_j)^2\,\int_{\mathbb{R}^d}\big(T(p+k)+T(p-k)-2T(p)\big)|\phi_j(p)|^2\;dp\\
    &-\sum_{\lambda_j\in J}(z-\lambda_j)\,\int_{\mathbb{R}^d}\big((T(p+k)-T(p))^2+(T(p-k)-T(p))^2\big)|\phi_j(p)|^2\;dp\\
    &\quad\quad=\\
    &\sum_{\lambda_j\in J}\sum_{\lambda_k\notin J}
    (z-\lambda_j)(z-\lambda_k)(\lambda_k-\lambda_j)\bigg(\bigg|\int_{\mathbb{R}^d}\phi_j(p-k)\phi_k^*(p)\;dp\,\bigg|^2
    +\bigg|\int_{\mathbb{R}^d}\phi_j(p+k)\phi_k^*(p)\;dp\,\bigg|^2\bigg).\\
    \end{split}
\end{equation}
One possibility to exploit this identity is to do a Taylor
expansion about $k=0$. We obtain the corresponding trace formula
for a self adjoint operator $G$ with $G$ being the generator of
the unitary group of translations in momentum space.
Indeed, for $C^2$ functions $T(p)$ we have
\begin{equation*}
    T(p+k)+T(p-k)-2T(p)= k(D^{(2)}T)(p)k+O(|k|^3)
\end{equation*}
and
\begin{equation*}
    (T(p+k)-T(p))^2+(T(p-k)-T(p))^2=2\big(\nabla T(p)k\big)^2+O(|k|^3),
\end{equation*}
and therefore, after division by $|k|^2$,
\begin{equation}\label{tf-U-integral-operator-k=0}
    \begin{split}
    &\sum_{\lambda_j\in J}  (z-\lambda_j)^2\,\int_{\mathbb{R}^d}k^{T}(D^{(2)}T)(p)k|\phi_j(p)|^2\;dp-2(z-\lambda_j)\,\int_{\mathbb{R}^d}\big((\nabla T(p))k\big)^2|\phi_j(p)|^2\;dp\\
    &\quad\quad =\\
    &2\sum_{\lambda_j\in J}\sum_{\lambda_k\notin J}
    (z-\lambda_j)(z-\lambda_k)(\lambda_k-\lambda_j)\bigg|\int_{\mathbb{R}^d}\phi_k^*(p)\nabla\phi_j(p)\;dp\,\bigg|^2.\\
    \end{split}
\end{equation}
If $T(p)=\alpha p^2$, then
\begin{equation*}
    T(p+k)+T(p-k)-2T(p)= 2\alpha k^2,\quad
\end{equation*}
and
\begin{equation*}
    (T(p+k)-T(p))^2+(T(p-k)-T(p))^2=8\alpha^2(pk)^2+2\alpha^2k^4.
\end{equation*}
If $T(p)=\sqrt{p^2+m^2}-m$, then $T(p)^2=p^2-2mT(p)$, and
\begin{equation*}
    T(p+k)+T(p-k)-2T(p)= \frac{m^2k^2+p^2k^2-(kp)^2}{\sqrt{p^2+m^2}^3}+O(|k|^3),\quad
\end{equation*}
and
\begin{equation*}
    (T(p+k)-T(p))^2+(T(p-k)-T(p))^2=\frac{2(pk)^2}{{p^2+m^2}}+O(|k|^3).
\end{equation*}
For related work on this relativistic kinetic energy operator we
refer to \cite{HaYi}.

\section{Examples}\label{Examples}
The study of the distribution of positive quadratic forms on
vectors of integers is sometimes referred to as the geometry of
numbers (e.g., \cite{Skr}).  The values of any such quadratic form
compose the spectrum of the Laplacian on a certain flat torus,
the dimensions of which determine the coefficients of the
quadratic form, or conversely.  Therefore the inequalities of the
preceding section have direct implications for the geometry of
numbers.  We begin this section by presenting those.

We note that $\sum^d_{\alpha=1} T_{{\bf \hat e}_\alpha j} = 2d+
8\|\nabla u_j\|^2=2d+8T_j$. In the case of the Laplacian,
$T_j=\lambda_j$ and $\lambda_j=4\pi^2\sum^d_{\alpha=1} n^2_{\alpha
j}$. Exploiting the abstract gap formula of Theorem
\ref{gap-condition}, after some simplification we find the following
inequality:
\begin{equation}\label{quadratic-ineq-1d-periodic}
   P_{2,N}(z):= \sum_{j=1}^{N}(z-\lambda_{j})(z-\frac{d+4}{d}\lambda_{j}-g)\leq
    N(z-\lambda_{N})(z-\lambda_{N+1})
\end{equation}
for all $z\in[\lambda_N,\lambda_{N+1}]$. We recall that $g :=
\frac{1}{d} \sum_{\alpha=1}^d{q_{\alpha}^2}$.

Analyzing the foregoing inequality following \cite{HaSt1}, we get

\begin{equation*}
   \bigg( \frac{\lambda_{N+1}-\lambda_{N}}{2}\bigg)^2\leq
   D_N:=\bigg(\frac{d+2}{d}\overline{\lambda_N}+\frac{g}{2}\bigg)^2-\frac{d+4}{d}\overline{\lambda_N^2}-g\overline{\lambda_N}
\end{equation*}
and
\begin{equation*}
   \frac{d+2}{d}\overline{\lambda_N}+\frac{g}{2}-\sqrt{D_N}\leq
   \lambda_{N}\leq \lambda_{N+1}\leq
   \frac{d+2}{d}\overline{\lambda_N}+\frac{g}{2}+\sqrt{D_N}.
\end{equation*}

As a first illustration we consider the case $d=1$. Obviously, we
want to choose $g$ as small as possible and the best choice is the
first nontrivial eigenvalue of the periodic Laplacian, i.e.
$g=4\pi^2$. Let $n$ ne a natural number and set $N:=2n+1$. Then
$\lambda_N=n^2$ and $\lambda_{N+1}=(n+1)^2$, which means that there
is a gap. We easily verify that in this case
\begin{equation*}
    D_N=\bigg(
    \frac{\lambda_{N+1}-\lambda_{N}}{2}\bigg)^2=\pi^2N^2.
\end{equation*}
Consequently, the quadratic polynomials (in $z$) on the right and left
sides of \eqref{quadratic-ineq-1d-periodic}
coincide for these values of $N$. Since (see also
\eqref{R-P-relation})

\begin{equation*}
    \frac{d}{dz}\;(z+\pi^2)^{-2-\frac{1}{2}}\sum_{j=1}^N(z-\lambda_j)^{2}=-\frac{1}{2}\;\sum_{j=1}^{N}(z-\lambda_{j})(z-5\lambda_{j}-4\pi^2)(z+\pi^2)^{-3-\frac{1}{2}},
\end{equation*}
we conclude that the nondecreasing function
\begin{equation*}
    z\mapsto (z+\pi^2)^{-2-\frac{1}{2}}R_2(z)
\end{equation*}
has critical points at the eigenvalues $\lambda_{j}$. Therefore
the positive shift $gd/4=\pi^2$ in $z$ cannot be replaced by any
smaller shift without losing the monotonicity property.

Next consider the two-dimensional Laplacian with periodic boundary
conditions on the square $Q=[0,2\pi]\times [0,2\pi]$. Its
eigenvalues are $m^2+n^2,m,n\in\mathbb{Z}$ with corresponding
eigenfunctions $\phi_{m,n}(x)=(2\pi)^{-1}\exp(imx+iny)$. The
counting function $N=N(x)$ counts the number of lattice points
inside the disc $D_x$ of radius $\sqrt{x}$ centered at the origin,
a sharp estimate of 
which is known in the literature on lattice points as the
{\it Gauss circle
problem} (see e.g. \cite{Kra}). Here we
only consider the bounds obtained from Theorem \ref{monotone1} and
the general inequality \eqref{quadratic-ineq-1d-periodic},
respectively. We follow \cite{Kra}, where, in place of
$N(x)$ the common notation is
\begin{equation*}
    R(x):= \# \{(m,n)\in\mathbb{Z}\times \mathbb{Z}: m^2+ n^2\leq x\},
\end{equation*}
which counts the lattice points inside the disk $D_x$ of
radius $x$ centered at the origin. Since
\begin{equation*}
    \sum_{m^2+n^2\leq x}
    f(m^2+n^2)=R(x)f(x)-\int_0^xf'(t)R(t)\;dt,
\end{equation*}
it follows that
\begin{equation*}
   R_2(x)=2\int_0^x(x-t)R(t)\;dt.
\end{equation*}
The bound of Theorem \ref{monotone1} reads as follows:
\begin{equation*}
    R_2(x)\leq \frac{\pi}{3}(x+\frac1{2})^3.
\end{equation*}
Defining, as in \cite{Kra}, the fluctuation about the Weyl
asymptotics by
\begin{equation*}
    R_2(x)-\frac{\pi}{3} x^3=2\Delta_2(x),
\end{equation*}
we find that
\begin{equation*}
    \Delta_2(x)\leq \frac{\pi}{48}(12x^2+6x+1),
\end{equation*}
which has to be compared with the
standard asymptotic estimate \cite{Kra}
\begin{equation*}
    \frac{|\Delta_2(x)|}{x^{\frac{5}{4}}}\leq C
\end{equation*}
for some positive constant $C$. Our estimate is only one sided and
too crude for large $x$.

Finally, we test the sharpness of our Lieb-Thirring inequalities for periodic
Schr\"{o}dinger operators. Consider the case of
$H(\alpha)=-\alpha\Delta -\gamma$ on $[0,2\pi]$ with periodic
boundary conditions, where $\gamma$ is some positive constant. Its
eigenvalues are $\lambda_j=\alpha j^2-\gamma$ with
$j\in\mathbb{Z}$. By Theorem \ref{LT}, the function
\begin{equation*}
    \alpha\mapsto \sqrt{\alpha}\sum
    (z-\frac{\alpha}{4}-\lambda_j)_{+}^2
\end{equation*}
is nondecreasing and therefore by \eqref{LTPeriodic} the following
Lieb-Thirring inequality holds:
\begin{equation*}
    \sqrt{\alpha}\sum
    (z-\alpha j^2+\gamma)_{+}^2\leq
    \frac{16}{15}(-\gamma-\frac{\alpha}{4}-z)_{-}^{\frac{5}{2}}.
\end{equation*}
In particular, taking $z=0$ we have
\begin{equation*}
    \sum (\frac{\gamma}{\alpha}- j^2)_{+}^2\leq
    \frac{16}{15}(\frac{\gamma}{\alpha}+\frac{1}{4})^{\frac{5}{2}}.
\end{equation*}
As in our first example we see that the shift $\alpha/4$ in $z$
cannot be replaced by any smaller shift without losing the
monotonicity property (choose $\gamma/\alpha=m^2$ for an integer
$m$). The presence of the shift is due to the zero eigenvalue.
Indeed, if $\gamma/\alpha<1$ we have
\begin{equation*}
  \big(\frac{\gamma}{\alpha}\big)^2\leq
  \frac{16}{15}(\frac{\gamma}{\alpha}+\frac{1}{4})^{\frac{5}{2}}
\end{equation*}
and without a shift this inequality clearly cannot be true.

\subsection*{Acknowledgments}
We are grateful to the Centre Interfacultaire Bernoulli of the EPF
Lausanne and the Mathematisches Forschungsinstitut Oberwolfach for
hospitality and to Lotfi Hermi and Rupert Frank for remarks and
references.

\end{document}